\documentclass{siamart171218}

\usepackage{amsmath} 
\usepackage{amssymb}  
\usepackage{graphics}
\usepackage{epsfig}
\usepackage{color}
\usepackage[ruled,algo2e]{algorithm2e}

\newtheorem{remark}{Remark}[section]
\newtheorem{example}{Example}[section]
\newtheorem{assumption}{Assumption}[section]
\newcommand{\beq}{\begin{equation}}
\newcommand{\eeq}{\end{equation}}

\newcommand{\RR}{{\mathbb{R}}}
\newcommand{\CC}{{\mathbb{C}}}

\def\eps{\varepsilon}

\newcommand{\Id}{{\rm I}}

\newcommand\sym{\mathrm{Sym}}
\newcommand\skews{\mathrm{Skew}}

\newcommand{\W}{\mathbb{W}}

\def\bng{\color{black}}
\def\eng{\color{black}}

\title{On the existence and approximation of a dissipating feedback\thanks{Version of September
29, 2019.}}

\author{Nicola Guglielmi\footnotemark[1] \and Valeria Simoncini\footnotemark[3]}

\begin{document}

\maketitle

\renewcommand{\thefootnote}{\fnsymbol{footnote}}
\footnotetext[1]{Gran Sasso Science Institute,
via Crispi 7, I-67010    L' Aquila,  Italy. 
Email: {\tt nicola.guglielmi@gssi.it}}
\footnotetext[3]{Dipartimento di Matematica, 
Alma Mater Studiorum - Universit\`a di Bologna,
Piazza di Porta San Donato 5,
       I-40127 Bologna,
       Italy, and IMATI-CNR, Pavia, Italy. 
Email: {\tt valeria.simoncini@unibo.it}}
\renewcommand{\thefootnote}{\arabic{footnote}}

\begin{abstract}
Given a matrix 
$A\in \RR^{n\times n}$ and a
tall rectangular matrix $B \in \RR^{n\times q}$, $q < n$, 
we consider the problem of making the pair $(A,B)$ dissipative, that is the determination
of a {\it feedback} matrix $K \in \RR^{q\times n}$ such that the field of values of
$A-B K$ lies in the left half open complex plane. We review and expand 
classical results available in the literature on the existence and parameterization
of the class of dissipating matrices, and we explore new matrix properties associated with 
the problem.
In addition, we discuss various computational strategies for approximating
the minimal Frobenius norm dissipating $K$.
\end{abstract}



\pagestyle{myheadings}
\thispagestyle{plain}
\markboth{N.~Guglielmi and V.~Simoncini}{Matrix dissipation: analysis and computations}

\section{Introduction}
A linear dynamical system 
\begin{equation}
\label{eq:ds}
\dot x = A x 
\end{equation}
with $A\in\RR^{n\times n}$ is said to be dissipative if the matrix $A$ has a field of values 
$W(A) = \{z \in \CC : z=x^* A x, x \in \CC^{n}, \| x \| = 1 \}$ contained in the left half
open complex plane\footnote{Analogously, a system is said to be weakly dissipative if $W(A)$ 
 is contained in the closed left half complex plane $\bar{\CC}^-$, which includes the imaginary axis,
and $W(A) \cap {\rm i} \RR \neq \emptyset$.}
$\CC^-$.
Here $x^*$ stands for the conjugate transpose of the complex vector $x$, and $\|x\|$ is the Euclidean norm. 
Dissipativity implies passivity, that is the property that a system requires no external energy to operate, thus the problem of transforming a 
non-passive system into a passive one by means of controls, plays an important role. 
{\color{black} Here we consider a control in the form $u=K x$ and a control matrix 
$B \in \RR^{n\times q}$ and $q<n$ (typically $q \ll n$), giving rise to the dynamical system
\begin{equation}\label{eqn:dcs}
\left\{ \begin{array}{rcl}
\dot x & = & A x - B u \\
u & = & K x 
\end{array}
\right.
\end{equation} 
Our aim is to find a matrix $K$ such that the system becomes dissipative.}
\eng
We shall call such a $K$ a ``dissipating feedback matrix''.
{\color{black} Dissipativity is an important property to be preserved in the
dynamical system. If possible, representations of the physical 
system that are naturally dissipative have
attracted significant interest, also in the very recent literature; this is the case,
for instance, for the port-Hamiltonian representation;
see e.g. \cite{vdS06,JZ12,GS18,BMV19}. Here we assume that only the data $A, B$ in
(\ref{eqn:dcs}) are available for a dissipativity analysis.} 
%
%
\eng

The considered problem is of great relevance 
in many applications; see, e.g., \cite{HP10,WMcK07}.
Several interesting examples are described in \cite{L79}, while linear models for real life mechanical, 
electrical and electromechanical control systems are considered in \cite{FPEN86}. 

%

%
\subsection{The problems}
We are interested in the problem of finding a (possibly weakly) 
dissipating feedback matrix $K$
to a {\it non-dissipative} linear control system of the form \eqref{eqn:dcs}.
%
Hence, the existence of a dissipating feedback matrix $K$ ensures
that the {\it closed-loop} linear system $\dot x = (A  - B K) x$ is dissipative. 
The feedback matrix $K$ is called {\it weakly} dissipating 
if $W(A-BK) \subset \bar{\CC}^-$ and $W(A-BK) \cap {\rm i} \RR \neq \emptyset$.
For real data, weak dissipativity
clearly implies that the field of values boundary passes through
the origin. 

In matrix terms the first problem can be stated as follows:

\bng
{{\bf P1.} \it Given $A \in \RR^{n\times n}$ and $B \in \RR^{n\times q}$, with
$q<n$, find 
a matrix $K \in \RR^{q \times n}$ such that the field of values of $A-BK$ is contained
in the left half open (closed) complex plane.} 
\eng

Throughout we assume that $A$ is stable, that is its eigenvalues are all in $\CC^-$, 
however $W(A)$ has nonzero intersection with the 
right half open complex plane.  


The problems of existence and representation of a feedback matrix
has been extensively analyzed in the control community; a widely used result
stated as \cite[Theorem 2.3.12]{Skeltonetal.98}
ensures the existence of $K$ under hypotheses on the data, while providing a parameterization
of all dissipating matrices. 
We revisit this parameterization, and observe that it 
may not include all possible feedback matrices. By using an alternative
proof of their existence, we thus propose alternative
parametrizations of dissipating matrices, and highlight the
actual degrees of freedom associated with the problem.

The concept of dissipativity is tightly related to other definitions of stability, which
are largely investigated in the Control community.
For real data,  dissipativity of $A$ corresponds to ensuring that $\frac 1 2 \lambda_{\max}( A +A^T) <0$,
where $\lambda_{\max}(\cdot)$ is the rightmost eigenvalue of the argument matrix.
Weak dissipativity requires that $\frac 1 2 \lambda_{\max}( A +A^T) \le 0$.
The quantity $\mu_2(A) = \frac 1 2 \lambda_{\max}( A +A^T)$ is called the numerical abscissa
(see, e.g., \cite{D59,S06}),
and it monitors the exponential stability (alas contractivity) property of the system solution 
$x(t)$, since it holds
$$
\|x(t_2)\| \le e^{\mu_2(A) (t_2-t_1)} \|x(t_1)\|.
$$
Clearly, if $\mu_2(A) \le 0$, then $\|x(t_2)\| \le \|x(t_1)\|$ and the system
is said to be exponentially stable.
In particular, concepts like $(M,\beta)$-stability are introduced and characterized 
(see e.g. \cite{HPW02,PP92}), meaning - for a given matrix $G$ - 
that  $\| {\rm e}^{G t} \| \le M {\rm e}^{\beta t}$.
For the system \eqref{eqn:dcs} the matrix is $G=A-BK$.
In our setting,  if $K$ is such that the field of values $W(A-BK)$ is all in $\CC^-$,
then  the system is $(1,\mu_2(A-B K))$-stable.  

The feedback matrix $K$ may be required to have additional properties, such
as a small norm. \bng Thus we also consider the problem: 

{{\bf P2.} \it Given $A \in \RR^{n\times n}$ and $B \in \RR^{n\times q}$, with
$q<n$, find 
a matrix $K \in \RR^{q \times n}$ of minimal norm such that the field of values 
of $A-BK$ is contained in the left half open (closed) complex plane.} 
\eng

{\color{black} For P2 we will consider minimization both in the Frobenious norm and
in the matrix norm induced by the vector Euclidean norm.}
Following standard approaches, we formulate 
problem {\color{black}P2} as an optimization procedure with inequality matrix constraints, thus
falling into a {\it linear matrix inequalities} (LMI) framework \cite{BEFB94,Skeltonetal.98}.
As an alternative we explore the use of a functional approach, which is a
variant of the method recently proposed in \cite{GL17}. Numerical experiments
on selected data illustrate the performance of tested methods.

In addition to the notation already introduced, the following definitions will be
used throughout.
 $I_n$ denotes the identity matrix of dimension $n$, and the subscript is removed
when clear from the context. For a square matrix $M$, 
$\sym(M) = (M+M^T)/2$ denotes its symmetric part
\bng
and $\skews(M) = (M-M^T)/2$ denotes its skew-symmetric part.
\eng
We denote by $\|\cdot\|_F$ the Frobenius norm on $\RR^{n\times n}$ and by 
$\langle X,Y \rangle = \mathrm{trace}(X^{\rm T} Y)= 
\sum_{i,j=1}^n  x_{ij} y_{ij}$ the corresponding inner product. Moreover,
$\|\cdot\|_2$ denotes the matrix norm induced by the Euclidean vector norm.

\section{Known existence results and parameterization}\label{sec:classics}
Conditions on the existence of a dissipating matrix
have been known for quite some time in the Control community. A
thorough and insightful discussion is available in the monograph \cite{Skeltonetal.98}.
The following fundamental theorem provides existence conditions for the matrix $K$ such that
$BKC+(BKC)^* + Q < 0$, for given $B, C$ and a symmetric $Q$ \cite[Theorem 2.3.12]{Skeltonetal.98}.

\begin{theorem}\label{th:Skelton}
Let the matrices $B\in\CC^{n\times q}$, $C\in\CC^{k\times n}$ and $Q=Q^*\in\CC^{n\times n}$ be
given. Then the following statements are equivalent:
\begin{itemize}
\item[(i)] There exists a matrix $K$ satisfying $BKC+(BKC)^* + Q < 0$.
\item[(ii)] The following two conditions hold
\begin{eqnarray*}
B^\perp Q (B^\perp)^* < 0 \quad or \quad BB^* > 0 \\
(C^*)^\perp Q ((C^*)^\perp)^* < 0 \quad or \quad C^* C > 0 .
\end{eqnarray*}
\end{itemize}
Suppose the above statements hold. Let 
$B=B_\ell B_r$, $C=C_\ell C_r$ be the full rank factorizations of $B$ and $C$, respectively. 
Then all matrices $K$ in statement (i) are given by
$$
K = B_r^+ H C_\ell^+ + Z - B_r^+ B_r Z C_\ell C_\ell^+
$$
where $Z$ is an arbitrary matrix and
\begin{eqnarray*}
H &:=& - R^{-1} B_\ell^* \Phi C_r^*(C_r\Phi C_r^*)^{-1} + S^{1/2} L (C_r\Phi C_r^*)^{-1/2} \\
S &:=& R^{-1} - R^{-1} B_\ell^*( \Phi - \Phi C_r^*(C_r\Phi C_r^*)^{-1}C_r\Phi) B_\ell R^{-1},
\end{eqnarray*}
where $L$ is an arbitrary matrix such that $\|L\| <1$ and $R$ is an arbitrary positive definite
matrix such that $\Phi : = (B_\ell R^{-1}B_\ell^* - Q)^{-1} > 0$.
\end{theorem}

Here $Q$ plays the role of $A+A^*$, so that item $(i)$ precisely corresponds to our setting.
In the theorem statement, $(B^\perp)^*$ is the matrix spanning the null space of $B^*$.
The theorem thus says that $K$ exists if and only if $Q$ is negative definite on the Kernel
of $B^*$ and on the Kernel of $C$, or otherwise, if $B$ ($C$) has full row (column) rank.
The theorem also provides a parameterization of dissipating matrices.

The following corollary specializes the result to our case, where $C=I$; see also
\cite[Corollary 2.3.9]{Skeltonetal.98}.

\begin{corollary}\label{cor:skelton}
Assume $C=I$ and $B$ full column rank, so that $B_\ell=B$ and $B_r=I$. 
With the notation of Theorem \ref{th:Skelton}, we have
$$
K = H = -R^{-1} B^* + R^{-\frac 1 2}L \Phi^{-\frac 1 2}, 
$$
where $L\in\RR^{q\times n}$ is an arbitrary matrix such that $\|L\| < 1$ and 
$R\in\RR^{q\times q}$ is an arbitrary positive definite
matrix such that $\Phi:= ( B R^{-1} B^* - Q)^{-1} >0$.
\end{corollary}

The role of the matrix $R$ is to push into the positive half complex plane the indefinite
 matrix $-Q = -(A + A^*)$. In \cite{Skeltonetal.98} the choice $R=\rho I_q$ for some large enough $\rho$
 is considered sufficient to be able to obtain a positive definite $\Phi$. However, by doing so,
some degrees of freedom may be lost. In particular, if one is interested in
a norm minimizing $K$, a full symmetric positive definite $R$ should be considered.

{By going through the proof of the previous theorem, it is possible to show
that there exist dissipating matrices that are not represented by the parameterization
$K$ given above.}
To this end, we first deepen our understanding of the quantities involved in
the classical parametrization in terms of invariant subspaces.
This will allow us to capture the role of the free matrices $R$ and $L$.
For $K=-R^{-1} B^* + R^{-\frac 1 2}L \Phi^{-\frac 1 2}$  we have\footnote{
Without loss of generality, 
to conform with the notation in \cite{Skeltonetal.98},
there is a change of
sign in the (1,1) block of the saddle point matrix, compared with ${\cal M}$ in (\ref{eqn:M})
later on.}
\begin{eqnarray*}
&& BK + K^* B^* + Q = [I, K^*] \begin{bmatrix} Q & B\\ B^* & 0 \end{bmatrix} 
\begin{bmatrix}I\\ K \end{bmatrix} \\
&=& \Phi^{-\frac 1 2} [ \Phi^{\frac 1 2}, -\Phi^{\frac 1 2} B R^{-\frac 1 2} + L^*  ] 
\begin{bmatrix} I & 0 \\ 0 & R^{-\frac 1 2} \end{bmatrix}
\begin{bmatrix} Q & B\\ B^* & 0 \end{bmatrix}
\begin{bmatrix} I & 0 \\ 0 & R^{-\frac 1 2} \end{bmatrix}
\begin{bmatrix} \Phi^{\frac 1 2} \\ - R^{-\frac 1 2} B^*\Phi^{\frac 1 2}  + L \end{bmatrix} \Phi^{-\frac 1 2}\\
&=& \Phi^{-\frac 1 2} [ I, -\Phi^{\frac 1 2} B R^{-\frac 1 2} + L^*  ] 
\begin{bmatrix} \Phi^{\frac 1 2} & 0 \\ 0 & R^{-\frac 1 2} \end{bmatrix}
\begin{bmatrix} Q & B\\ B^* & 0 \end{bmatrix}
\begin{bmatrix} \Phi^{\frac 1 2} & 0 \\ 0 & R^{-\frac 1 2} \end{bmatrix}
\begin{bmatrix} I \\ - R^{-\frac 1 2} B^*\Phi^{\frac 1 2}  + L \end{bmatrix}\Phi^{-\frac 1 2} \\
&=& 
\Phi^{-\frac 1 2} [ I, -\Phi^{\frac 1 2} B R^{-\frac 1 2} + L^*  ] 
\begin{bmatrix} \widetilde Q & \widetilde B\\ \widetilde B^* & 0 \end{bmatrix}
\begin{bmatrix} I \\ - R^{-\frac 1 2} B^*\Phi^{\frac 1 2}  + L \end{bmatrix}\Phi^{-\frac 1 2} 
\end{eqnarray*}
where $\widetilde Q = \Phi^{\frac 1 2} Q \Phi^{\frac 1 2}$ and $\widetilde B = \Phi^{\frac 1 2} B R^{-\frac 1 2}$.
Therefore,
$BK + K^* B^* + Q <0$ if and only if
$$
[ I, -\widetilde B  + L^*  ] 
\begin{bmatrix} \widetilde Q & \widetilde B\\ \widetilde B^* & 0 \end{bmatrix}
\begin{bmatrix} I \\ - \widetilde B^* + L \end{bmatrix} < 0 ,
$$
so that the block diagonal matrix with $\Phi^{\frac 1 2}$ 
and $R^{-\frac 1 2}$ provides an arbitrary scaling of the original
saddle point matrix; see also Remark \ref{rem:scaling} later on.
Since
$\Phi= ( B R^{-1} B^* - Q)^{-1}$, it follows that
\begin{eqnarray}\label{eqn:orth}
\widetilde B \widetilde B^* - \widetilde Q = \Phi^{\frac 1 2} (BR^{-1} B^* - Q) \Phi^{\frac 1 2} = I .
\end{eqnarray}
After simple algebra we can thus write
\begin{eqnarray}
[ I, -\widetilde B  + L^*  ] 
\begin{bmatrix} \widetilde Q & \widetilde B\\ \widetilde B^* & 0 \end{bmatrix}
\begin{bmatrix} I \\ - \widetilde B^* + L \end{bmatrix}  & = &
 (\widetilde Q - \widetilde B \widetilde B^*) + 
(L\widetilde B^* + \widetilde B L^* - \widetilde B \widetilde B^*) \nonumber\\
&=&
 -I + (L\widetilde B^* + \widetilde B L^* - \widetilde B \widetilde B^*) . \label{eqn:L}
\end{eqnarray}
The first matrix product is negative definite if and only if
 $-I + (L\widetilde B^* + \widetilde B L^* - \widetilde B \widetilde B^*)<0$.
Assume all data are real, and let
$x$ be such that $\|x\|=1$. If $\widetilde B^*x=0$ then
$-x^*x + 2x^*L\widetilde B^*x  - \|\widetilde B^*x\|^2 = -1 <0$ and the inequality is obtained.
If $\|B^* x\|\ne 0$, and under the hypothesis that $\|L\|<1$ we obtain
\begin{eqnarray*}
-x^*x + 2x^*L\widetilde B^*x  - \|\widetilde B^*x\|^2 &\le& 
-1 + 2 \|x^*L\| \|\widetilde B^*x\| - \|\widetilde B^*x\|^2 \\
&< &
-1 + 2 \|\widetilde B^*x\| - \|\widetilde B^*x\|^2  = -(\|\widetilde B^* x\| -1)^2  < 0.
\end{eqnarray*}
In summary, we see that the role of the $q\times q$ matrix
$R$ is to define the positive definite matrix $\Phi$ so
that (\ref{eqn:orth}) holds. 
The matrix $L$ yields the ``if'' statement. However, it seems that $L$ does not
necessarily need to have norm less than one for the desired inequality to be satisfied.
The following example illustrates one such case. In other words, the part of the statement
in Theorem \ref{th:Skelton} stating that all matrices $K$ have the given parametrization 
only considers a subset of all possible dissipating matrices.

\begin{example}
Consider $Q={\rm diag}(\alpha, -\alpha)$, with $\alpha>0$, and $B=e_1=[1;0]$. 
Let us take $R^{-1}=\widehat \alpha$ with $\widehat \alpha > \alpha$. Then
$$
\Phi = (BR^{-1} B^* - Q)^{-1} = {\rm diag}(\frac 1{\widehat\alpha - \alpha}, \frac 1{\alpha}) > 0, \quad
\widetilde B = \Phi^{\frac 1 2} B R^{-\frac 1 2} = 
\frac{\sqrt{\widehat \alpha}}{\sqrt{\widehat\alpha - \alpha}} e_1
$$
 with 
$\|\widetilde B\| = \frac{\sqrt{\widehat \alpha}}{\sqrt{\widehat\alpha - \alpha}} >1$
for all choices of $\alpha >0$ and $\widehat \alpha > \alpha$.
By taking $L = \frac 1 2 \widetilde B$ we can select $\alpha$ and $\widehat \alpha$
so that $\|L\| \ge 1$, while in (\ref{eqn:L}) for this choice of $L$ we have
 $-I + (L\widetilde B^* + \widetilde B L^* - \widetilde B \widetilde B^*) = - I$, which is
clearly negative definite.
\end{example}

\section{An invariant subspace perspective for the parametrization of the dissipating matrix}
In this section we provide a different perspective, that allows us to determine a richer
parametrization of dissipating matrices. We first restate the existence condition in terms
of an eigenvalue problem.
To this end we need to recall a standard result on structured (saddle point) matrices.

\vskip 0.1in
\begin{proposition}(\cite{CC84})\label{prop:spd}
If the matrix $-(A+A^T)$ is positive definite on the kernel of $B^T$,
then the matrix
\begin{eqnarray}\label{eqn:M}
{\cal M} = \begin{bmatrix} -(A+A^T) & B \\ B^T & 0 \end{bmatrix}
\end{eqnarray}
has exactly $n$ positive and $q$ negative eigenvalues.
\end{proposition}
\vskip 0.1in
We can state the existence result of a dissipating feedback matrix $K$ by using
a quite different proof, which sheds light into different properties of the matrix $K$.
In particular, similarities with the solution matrix of the Riccati equations can
be readily observed; see, e.g., \cite{S16} and references therein.

\vskip 0.1in
\begin{theorem}\label{th:exist}
The matrix $A+A^T$ is negative definite on the kernel of $B^T$ if and only if
there exists a matrix $K\in\RR^{q\times n}$ such that $W(A-BK) \subset {\mathbb C}^-$.
\end{theorem}

\begin{proof}
We first prove that if the condition on $A+A^T$ holds, then
there exists a matrix $K$ such that $W(A-BK) \subset {\mathbb C}^-$.

Proving that $W(A-BK) \subset {\mathbb C}^-$ for some $K$ corresponds
to stating that the symmetric matrix $(A-BK)+(A-BK)^T$  is
negative definite.
We can write
$$
(A-BK)+(A-BK)^T = (A+A^T) - BK - K^T B^T = 
-[I, K^T] \begin{bmatrix} -(A+A^T) & B \\ B^T & 0 \end{bmatrix}
 \begin{bmatrix} I\\ K \end{bmatrix} .
$$
Therefore, if $K$ is chosen so that the matrix $\cal M$ in (\ref{eqn:M})
is positive definite onto the space spanned by the columns of $[I;K]$, then
$(A-BK)+(A-BK)^T$ is negative definite.
Using Proposition \ref{prop:spd} it is possible to determine
an invariant subspace of ${\cal M}$ corresponding to the $n$ positive 
eigenvalues of ${\cal M}$. We next show that this gives the sought 
after matrix $K$.
Let the orthonormal columns $[X;Y]$ span
this invariant subspace, with $X\in\RR^{n\times n}$ and $Y\in\RR^{q\times n}$. 
Then we have
\begin{eqnarray}\label{eqn:eigM}
{\cal M} \begin{bmatrix} X\\Y\end{bmatrix} =
 \begin{bmatrix} X\\Y\end{bmatrix} \Lambda 
\end{eqnarray}
with $\Lambda$ diagonal and positive definite.
Moreover, multiplying from the left by $[X^T, Y^T]$ we can write
\begin{eqnarray}
\Lambda &= & [X^T,Y^T]{\cal M} \begin{bmatrix} X\\Y\end{bmatrix} = 
-X^T(A+A^T)X + X^T BY + Y^T B^T X \label{eqn:rel}\\
&=& X^T (-AX+BY) + (-X^TA^T+YB^T) X  =: S + S^T, \nonumber
\end{eqnarray}
where $S=X^T(-AX+BY)$. Since $\Lambda$ is positive definite, we
have that $S+S^T$ is also positive definite, that is the
field of values of $S$ is all in the positive right half  open complex
plane. In particular, this implies that $S$ is nonsingular, and
thus $X$ is nonsingular. 
Therefore we can define $K := YX^{-1}$. Then collecting $X$ and $X^T$
on both sides of the right-most expression in (\ref{eqn:rel}),
\begin{eqnarray*}
0 < \Lambda &=& X^T( -(A+A^T) +  BYX^{-1} + X^{-T}Y^T B^T) X \\
&= &
 X^T( -(A+A^T) +  BK + K^T B^T) X  
= 
X^T [ I, K^T] {\cal M} \begin{bmatrix} I\\K\end{bmatrix} X. 
\end{eqnarray*}
Since the eigenvalues of the two congruent matrices 
$$
X^T [ I, K^T] {\cal M} \begin{bmatrix} I\\K\end{bmatrix} X \quad {\rm  and}
\quad [ I, K^T] {\cal M} \begin{bmatrix} I\\K\end{bmatrix}
$$
 have the same sign, this implies that
$[ I, K^T] {\cal M} \begin{bmatrix} I\\K\end{bmatrix}$ is also positive definite.

\vskip 2mm
We finally prove the converse by negating that $-(A+A^T)$ is positive definite on
$\ker(B^T)$.
Suppose then that there exists $x \in \ker(B^T)$ such that $x^T (A+A^T) x \ge 0$,
then we have
\[
x^T ( (A+A^T ) - B K - K^T B^T ) x = x^T (A+A^T) x \ge 0 ,
\]
which means that $W(A-BK) \not\subset {\mathbb C}^-$
independently of $K$, completing the proof.
\end{proof}

The proof in constructive, since it determines one such $K$ explicitly.
Indeed, for small matrices 
a dissipating feedback 
matrix $K$ can be computed by first determining the eigenvector matrix $[X;Y]$
corresponding to all positive eigenvalues of ${\cal M}$, and then setting
$K= Y X^{-1} \in\RR^{q\times n}$.
%


\begin{remark}
{\it 
From its construction, it follows that $K=YX^{-1}$ is full (row) rank, equal
to $q$.
Indeed, we first notice that rank($K$) = rank($Y$). Moreover,
the second block row of (\ref{eqn:eigM}) yields $B^T X = Y \Lambda$.
Since both $\Lambda$ and $X$ are square and full rank, we obtain
rank($Y$) = rank($B$).
}
\end{remark}

\begin{remark}
{\it 
From the previous remark it also follows that since $B^T X \Lambda^{-1} = Y$ and
$X$ is nonsingular, we have $K=YX^{-1} = B^T X \Lambda^{-1} X^{-1}$, that is,
$K$ can be written as $K=B^T W$ for some nonsingular matrix $W$. Other strategies
discussed in the following will also determine a similar form, but with possibly 
singular $W$.
}
\end{remark}

\subsection{New parametrizations of dissipating matrices}
The parametrization in Corollary \ref{cor:skelton} depends on two matrices,
$R$ and $L$, giving at most $q(q+1)/2+nq$ degrees of freedom. However, by generalizing the
setting of our
Theorem \ref{th:exist}, we can see that dissipating matrices can be parametrized by
a larger number of degrees of freedom, therefore many more such matrices can be
defined than those introduced in Corollary \ref{cor:skelton}.

By generalizing the representation of Theorem \ref{th:exist}, we next present
two
different parametrizations of the possible families of dissipating feedback matrices.


\vskip 0.1in
\begin{proposition}\label{prop:nonunique}
Assume that the condition of Theorem \ref{th:exist} holds.
Let ${\cal M}=Q {\rm blkdiag}(\Lambda_+$, $\Lambda_-) Q^T$ be the eigendecomposition of
${\cal M}$, where $\Lambda_+$ ($\Lambda_-$) is diagonal with all the $n$ positive ($q$ negative)
eigenvalues of ${\cal M}$. 
Partition further $Q =[ Q_{11}, Q_{12}; Q_{21}, Q_{22}]$ with $Q_{11}\in \RR^{n\times n}$ nonsingular.
Then for any $H_2\in\RR^{q\times n}$ such that 
$\alpha=\|H_2(I-Q_{12}H_2)^{-1}Q_{11}\|^2$ satisfies
$\min \lambda(\Lambda_+) > \alpha \max \lambda(|\Lambda_-|)$, the feedback matrix
$$
K = Q_{21} Q_{11}^{-1} + (Q_{22}-Q_{21}Q_{11}^{-1} Q_{12}) H_2 ,
$$
is dissipating.
\end{proposition}
\vskip 0.1in

{The proof is postponed to the appendix.}

Proposition~\ref{prop:nonunique} shows that as long as it is possible to
separate the negative and positive eigenvalues of $\cal M$, a different matrix $K$ can
be obtained. Different values of $\alpha$ yield different values of $\|K\|$.

The result of Theorem \ref{th:exist} corresponds to using the limiting case
 $\alpha =0$ in Proposition \ref{prop:nonunique}, that is $H_2 = 0$ in the
definition of $Z$ in the proposition proof. This way, $K$ is well defined
as long as $\lambda_{\min}^+>0$, that is as long as
$\cal M$ has $n$ strictly positive eigenvalues, as indeed shown by Theorem \ref{th:exist}.
Indeed, the expression 
$K = Q_{21} Q_{11}^{-1} + (Q_{22}-Q_{21}Q_{11}^{-1} Q_{12}) H_2$ parametrizes
$K$ in terms of some matrix $H_2$ with the required conditions.
This parametrization may be used for determining the feedback matrix $K$ having
certain properties, such as  minimum Frobenius norm, see section~\ref{sec:min}.
%
Due to the low number of degrees of freedom, however, this parametrization is
unlikely to cover all possible feedback matrices $K$. This concern was
confirmed by some of our uumerical experiments, which showed that this procedure usually
determines a local minimum, which does not seem to be the global one. 

The next proposition provides another, more general parametrization for the set of dissipating feedback matrices,
by means of a pencil $({\cal M}, {\cal D})$, where $\cal D$ is a symmetric positive definite
matrix playing the role of the parametr. In particular, this means that 
at least $(n+q)(n+q-1)/2$ degrees of freedom
are available for the family of dissipating matrices.

\vskip 0.1in
\begin{theorem} \label{prop:D}
There exists a matrix $K$ such that $W(A-BK) \subset \CC^{-}$ if and only if
the pencil $({\cal M}, {\cal D})$ admits $n$ positive eigenvalues for some 
symmetric and positive definite matrix
 ${\cal D}\in\RR^{(n+q)\times (n+q)}$.
\end{theorem}
\vskip 0.1in

\begin{proof}
We first recall that
the signature of the eigenvalues of $({\cal M}, {\cal D})$ is the same as that of ${\cal M}$ 
\cite[Theorem 5]{W73}.

Assume there exists ${\cal D}$ symmetric and positive definite such that
${\cal M} [X;Y] = {\cal D} [X;Y] \Lambda$ with $\Lambda>0$, with $[X;Y]$ ${\cal D}$-orthogonal.
Since $0 < \Lambda = [X^T, Y^T] {\cal M} [X;Y]$, proceeding as in the discussion after (\ref{eqn:rel})
the nonsingularity of $X$ is ensured.
Finally, setting $K:=YX^{-1}$,
\begin{eqnarray*}
[I, K^T] {\cal M} \begin{bmatrix}I \\ K \end{bmatrix} &=&
X^{-T} [X^T, Y^T] {\cal M} \begin{bmatrix}X \\ Y \end{bmatrix} X^{-1} \\
&=& 
X^{-T} [X^T, Y^T] {\cal D} \begin{bmatrix}X \\ Y \end{bmatrix} \Lambda X^{-1}  =
X^{-T}  \Lambda X^{-1}  > 0 .
\end{eqnarray*}

We next prove that if $K$ exists such that $W(A-BK) \subset \CC^{-}$,
then we can define a  
 symmetric and positive definite matrix ${\cal D}$.
Let $U=[I;K]$ and define 
$$
 {\cal D} = U D U^T + U_\perp D_\perp U_\perp^T,
$$
with $D=(U^T U)^{-2}$,
$[U, U_\perp]$ square and full rank with $U_\perp ^T  U=0$, 
 and for any symmetric and positive definite matrix $D_\perp \in\RR^{n\times n}$.
By construction we have  $U^T {\cal D} U = I$. 
We have thus found a subspace of dimension $n$, range$(U)$, such that, for any $0\ne x\in\RR^n$,
$$
\frac{x^T U^T {\cal M} U x}{x^T U^T {\cal D} U x} >0
$$
which implies that the pencil $({\cal M}, {\cal D})$ has at least $n$ positive eigenvalues.
\end{proof}

As opposed to the case ${\cal D}=I$, it does not seem to be possible to ensure that
$K$ has full rank, because $Y$ and $X$ depend on the matrix ${\cal D}$ to be
determined.

Note also that $\cal D$ may also be viewed as
the matrix defining a different inner product associated with the invariant subspace basis.

\begin{remark}\label{rem:scaling}
Since the matrix ${\cal D}$ is somewhat arbitrary, except for being symmetric and
positive definite, a block diagonal matrix could be considered.
On the other hand, this simplifying strategy would significantly decrease the
number of degrees of freedom, which play a role when looking for the minimal norm
feedback matrix, as discussed in the next section.
A similar drawback can be observed for the classical derivation highlighted in the second
part of section \ref{sec:classics}: indeed, in there,
a scaling with the free parameter matrix diag($\Phi^{\frac 1 2}, R^{-\frac 1 2})$ is performed,
but this may prevent the parametrized family from containing the matrix of minimal norm.
\end{remark}


\section{Computing a (weakly) dissipating feedback of minimal norm}\label{sec:min}
In this section we \bng address Problem P2 \eng and explore the possible computation of
a feedback matrix of minimal norm that makes the system either dissipative
or weakly dissipative. Let $\W^{q\times n} (A,B)$ be the set
of weakly dissipating matrices for the pair $(A,B)$.
The problem can thus be stated as: 

{\it Find $K_*\in\W^{q\times n}(A,B)$ such that}
\begin{eqnarray}\label{eqn:pbK}
\inf_{K\in\W^{q\times n}(A,B)} \| K \|_{\star} .
\end{eqnarray}
Here $\|\cdot\|_{\star}$ stands for the Frobenius norm ($\|\cdot\|_F$)
or the 2-norm ($\|\cdot\|_2$).  
 
The following result implies that the feedback matrix of minimal norm is 
to be found among the weakly dissipating matrices.

\begin{proposition} \label{prop:weak}
Assume that $W(A) \cap \CC^+ \neq \emptyset$ and let $K_1$ be a dissipating feedback matrix.
Then there exists a weakly dissipating feedback matrix $K_2$ with $\| K_2 \|_{\star} < \| K_1 \|_{\star}$.
\end{proposition} 
\begin{proof}
Let $K(\rho) = (1-\rho) K_1$, $0 \le \rho \le 1$. Naturally $W(A-B K(0)) = W(A - B K_1) \subset {\CC}^-$
and $W(A- B K(1)) = W(A) \not\subset \bar{\CC}^-$.

By continuity of eigenvalues of $S(\rho) = \sym(A -B K(\rho))$ we have that for sufficiently small $\rho>0$,
$\eta \left( \sym(A -B K(\rho)) \right) < 0$ and there exists $\rho = \rho_0 > 0$ such that $\eta \left( \sym(A -B K(\rho_0)) \right)=0$.
Setting $K_2 = (1-\rho_0) K_1$ determines a weakly-dissipating feedback $K_2$ with 
$\| K_2 \|_{\star} = (1-\rho_0)\,\| K_1 \|_{\star}$.  
\end{proof}

The following result is concerned with the existence of a weakly dissipating
minimizer for \eqref{eqn:pbK}.

\begin{proposition}\label{prop:wp}
Assume that $A +A ^T$ is negative definite on the kernel of $B^T$.
Then \eqref{eqn:pbK} is equivalent to
\begin{eqnarray}\label{eqn:pbKwk}
\min_{K\in\W^{q\times n} (A,B)} \| K \|_{\star} .
\end{eqnarray}
\end{proposition}

\begin{proof}
Under the considered assumption, Theorem~\ref{th:exist} implies the existence
of a dissipating matrix $K_1$, then the set $\W^{q\times n} (A,B)$ is not empty.
Moreover, Proposition~\ref{prop:weak} implies the existence of a weakly dissipating 
matrix $K_2 \in \W^{q\times n} (A,B)$ with $\alpha := \| K_2 \|_{\star} \le \| K_1 \|_{\star}$.
Thus we can look for the solution to \eqref{eqn:pbK} in the bounded and closed 
(and thus compact) set 
$$\{ K \in\W^{q\times n}(A,B) \ {\rm s.t.} \ \| K \|_{\star} \le \alpha \}.$$
%
%
Since $\| \cdot \|_{\star}$ is a continuous function, the result follows from
Weierstra{\ss} Theorem.
\end{proof}

Note that in the case where one wishes to compute some strictly dissipating feedback
it would be sufficient to replace the matrix $A$ by  
$A_\delta := A + \delta I$, where $\delta$ represents the maximal real part of $W(A-B K)$.
Then applying the same procedure to the pair $\{A_\delta,B\}$ provides 
a strictly dissipating feedback.

Before we proceed with the actual computational strategies, we linger over some
spectral properties of the involved matrices. 

\begin{proposition}\label{prop:rankt}
Assume that $\sym(A )$ has $t$ positive eigenvalues with corresponding
eigenvectors $Q_-=[q_1, \ldots, q_t]$, and that $K\in\RR^{q\times n}$
is a dissipating feedback. Then it must be rank$(Q_-^T(B+K^T)) \ge t$.
\end{proposition}

\begin{proof}
See Appendix B.
\end{proof}

We next show that in correspondence to a weakly dissipating matrix $K$ there is
a nontrivial null space of $\sym(A -BK)$ of dimension at most $q$.

\begin{proposition}\label{prop:rankm}
Assume that $A +A ^T$ is negative definite on the kernel of $B^T$.
If $K$ is a weakly dissipating feedback 
then $\sym(A - BK)$ has a zero eigenvalue with multiplicity $m$, 
with $0<m\le q$.
\end{proposition}

\begin{proof}
Using Proposition \ref{prop:weak}, the hypothesis ensures that there
exists a weakly dissipating matrix $K$.
We only need to show that $\sym(A-BK)$ has at most $q$ zero eigenvalues.
Let $[B_0, N]$ be unitary, with Range$(B_0) = $Range$(B)$, so that Range$(N)$ is the null space of  $B^T$.
Let $K$ be weakly dissipating, so that
$( (A +A ^T) - BK - K^T B^T)$ has $m>0$ zero eigenvalues.
We can write
$$
[B_0, N]^T ( -(A +A ^T) + BK + K^T B^T) [B_0, N] =
\begin{bmatrix} 
S_{11} & S_{12} \\ S_{12}^T & S_{22}
\end{bmatrix} =: {\cal S},
$$
with $S_{22} \in \RR^{(n-q)\times (n-q)}$ and $S_{12} \in \RR^{q\times (n-q)}$.
By hypothesis it follows that $S_{22} = - N^T(A +A ^T)N >0$.
Let $[u;v]$ be a nonzero vector such that ${\cal S} [u;v]=[0;0]$. Then it must hold
that $v = - S_{22}^{-1} S_{12}^T u$ with $u\in\RR^{q}$. The eigenspace of ${\cal S}$
associated with the zero eigenvalue is thus spanned by the vectors
$[I_q; - S_{22}^{-1} S_{12}^T] u$, and there are thus at most $q$ of them, that are
linearly independent, that is there are $m\le q$ zero eigenvalues.
\end{proof}

%
%

\subsection{The LMI framework}
The problem (\ref{eqn:pbK}) can be stated as the following LMI optimization
problem. 
Following standard strategies (see, e.g., \cite{BEFB94}),
if the 2-norm is to be minimized, then the problem can be stated as
\begin{eqnarray}\label{eqn:LMI_2norm}
\min_{K\in\RR^{q\times n}} \|K\|_2 && \quad
{\rm subject \, \, to} \\
A+A^T  - B K  - K^T B^T \le 0, && \quad
\begin{bmatrix}\gamma I_q  & K \\ K^T & \gamma I_n \end{bmatrix} \ge 0 
\end{eqnarray}
where $\gamma>0$ is such that $\|K\| \le \gamma$. The problem is thus expressed
in terms of the two variables $K$ and $\gamma$, the first of which is a
rectangular matrix.

If the Frobenius norm is to be minimized, 
the problem becomes
\begin{eqnarray}\label{eqn:LMI_Fnorm}
\min_{K\in\RR^{q\times n}} \|K\|_F && \quad
{\rm subject \, \, to} \\
A+A^T  - B K  - K^T B^T \le 0, && \quad 
\begin{bmatrix}I & {\rm vec}(K) \\ {\rm vec}(K)^T & \gamma  \end{bmatrix} \ge 0 
\end{eqnarray}
where ${\rm vec}(K)$ stacks all columns of $K$ one after the other,
so that $\|K\|_F^2 \le \gamma$; see, e.g., \cite{D17}.

Both problems can be numerically
solved by using standard LMI packages. In our computational
experiments we used the Matlab version of Yalmip with the call to
either SeDuMi (see \cite{Sedumi}) or Mosek (see \cite{Mosek}). Some of these results
are reported in section~\ref{sec:expes}.

\subsection{A direct approach}

Using Theorem \ref{prop:D} we can compute the feedback matrix $K=YX^{-1}$ of minimal 
norm by solving the following optimization problem: 
\begin{equation}\label{Opt1}
\left\{
\begin{array}{c}
\inf\limits_{{\cal D} > 0} \| Y X^{-1} \|_F
 \\[0.25cm]
 { subject \ to}
\\  [0.3cm]
{\cal M} [X;Y] = {\cal D} [X;Y] \Lambda \quad \mbox{\it with} \quad \Lambda > 0. 
\end{array}
\right.
\end{equation}

This method has limitations when applied to problems of large dimensions,
that is when $n \gg 1$, moreover it seems to strongly depend on
the starting guess, as many local minima seem to exist.


\subsection{A gradient system approach} \label{sec:GL}
In this section we propose a gradient-flow differential equation approach that adapts 
to our setting a strategy first proposed in \cite{GL17}.
Given the matrix $\sym(A)$ and identifying its $m$ rightmost eigenvalues 
(e.g. its positive eigenvalues),
we construct a smoothly varying matrix $K$ that moves these eigenvalues to the origin,
so as to make the system weakly dissipative. We look for one such feedback matrix $K$ 
having minimum Frobenius norm.
%
%
We write $K=\eps E$ with $E$ of unit Frobenius norm, and with perturbation size $\eps>0$.
For a fixed $\eps>0$, we minimize the function
\begin{equation}
\label{F-eps}
F_\eps(E) = \frac12 \sum_{i=1}^m \Bigl( \lambda_i\left( \sym(A -\eps B E) \right)  \Bigr)^2 
\qquad\quad \hbox{constrained by }\|E\|_F=1,
\end{equation}
by solving numerically the corresponding gradient-flow differential equation. Here
$\lambda_i$s are the $m$ rightmost eigenvalues of the argument symmetric matrix.
We denote the obtained minimum by 
$E_\eps$ and then look for the smallest $\eps>0$ such that $F_\eps(E_\eps)=0$, which we denote by $\eps_m^*$.  
In general, the existence of $\eps_m^*$ is not guaranteed.
Formally, this can be expressed as: 

{\it Solve}
\begin{eqnarray}\label{eqn:minmin}
\min_{\eps>0} \min_{E\in\RR^{q\times n}\atop \|E\|_F=1} F_\eps(E) .
\end{eqnarray}
Clearly, the minimum of $F_\eps(E)$ is zero, that is with the optimal $K=\eps_m^* E$
the matrix $\sym(A  - BK)$ has $m$ coalescent eigenvalues.
 
Due to classical results on eigenvalue interlacing of low-rank modifications of symmetric matrices
\cite{HJ13},  
the number of positive eigenvalues of $\sym(A )$ provides a rigorous lower bound 
for rank$(K)$ in order to find an optimal weakly dissipating feedback.

The two-phase method works as follows.

{\it Inner procedure}. Assume $\eps>0$ is fixed.
Suppose that $E(t)$ is a smooth matrix-valued function of $t$ such that the $m$ largest eigenvalues 
of $\sym(A  - B \eps E(t))$, denoted by $\lambda_i(t)$ for $i=1,\dots,m$,
are simple with corresponding eigenvectors $x_i(t)$ normalized to have unit $2$-norm. 
Define $G(E)={}-\sum_{i=1}^m  \lambda_i z_i x_i^{T} = - Z D X^T$, with $z_i = B^{T} x_i$.
The steepest descent direction $\dot{E}$ for the functional $F_\eps(E)$ is obtained by solving
the gradient system (see \cite{GL17})
\begin{equation} \label{ode-E}
\dot{E} = -G(E) + \beta E, 
\qquad\hbox{with }\ \beta= \langle G(E),E\rangle .  
\end{equation}
Note that $G(E)$ is the free gradient matrix of $F_{\eps}(E)$.
Then the following result generalizes the corresponding theorem in \cite{GL17}.

\begin{theorem}\label{thm:stat}
The following statements are equivalent along solutions of  {\rm (\ref{ode-E})}, provided that the 
$m$ largest eigenvalues $\lambda_i$ of $\sym(A  - \eps B E(t))$ are simple and that
there exists at least an index $i\le m$ such that $\lambda_i \ne 0$.
\begin{enumerate}
\item $\frac{d}{dt} F_\eps\bigl(E(t)\bigr) = 0$.
\item $\dot E = 0$.
\item $E$ is a real multiple of $G(E)$. 
\end{enumerate}
\end{theorem}
The proof follows the same lines as that of \cite[Theorem 3.2]{GL17}. 

{
Since the equilibrium of the ODE \eqref{ode-E} has rank-$m$, we proceed
similarly to \cite[equation (19)]{GL17}  and replace the matrix differential 
equation \eqref{ode-E} on $\RR^{q\times n}$ by a projected differential equation
onto the manifold of rank-$m$ matrices, so as to maintain the solution equilibria.
%
To preserve the projection property in the numerical treatment, 
we have considered a projected Euler method on the manifold  of rank-$m$
matrices (see, e.g., \cite[section IV.4]{HLW06}).
}

{\it Outer procedure.}
We let $E(\eps)$ of unit Frobenius norm be a local minimizer of the 
inner optimization problem in \eqref{eqn:minmin} and for $i=1,\dots,m$ we denote by 
$\lambda_i(\eps)$, $x_i(\eps)$ and $z_i(\eps)$ the corresponding largest eigenvalues, eigenvectors and
$z$-vectors of $\sym(A  - \eps B E(\eps))$. 
Finally we let $\eps_m^*$ be the smallest value of $\eps$ such that $F_\eps(E(\eps))=0$. 

To determine $\eps_m^*$, we are thus left with a one-dimensional root-finding problem, for 
which a variety of standard methods are available. Following \cite{GL17} in our
implementation we have used a Newton-like algorithm in the form
$$
\eps_{k+1} = \eps_{k} - \frac{f(\eps_k)}{f'(\eps_k)},
$$
where $f(\eps)=F_\eps(E(\eps))$ and ${\phantom{a}'}= d/d\eps$.
To use this iteration we need to impose the following
extra assumption, which is not restrictive in practice.

\begin{assumption}
For $\eps$ close to $\eps_m^*$ and $\eps<\eps_m^*$, 
we assume that the $m$ largest eigenvalues  of $\sym(A  - \eps B E(\eps))$  
are \emph{simple} eigenvalues. Consequently $E(\eps)$ and these eigenvalues 
are smooth functions of $\eps$, as well as the associated vectors $x_i(\eps), z_i(\eps)$.
\label{assumpt}
\end{assumption}

Then under Assumption~{\rm \ref{assumpt}} 
the function $f(\eps)$  
is differentiable and its derivative equals (see, \cite[Lemma 3.5]{GL17})
$$
f'(\eps) =  - \| G(\eps) \|_F.
$$ 
Since the eigenvalues are assumed to be simple,
the function $f(\eps)$ has a double zero at $\eps_m^*$ because it is a sum of squares, and hence
it is convex for $\eps \le \eps_m^*$.
This means that we may approach $\eps_m^*$ from the left by the classical Newton iteration,
which satisfies
$|\eps_{k+1} - \eps_m^*| \approx \frac12 |\eps_{k} - \eps_m^*|$ 
and $\eps_{k+1} < \eps_m^*$ if $\eps_k < \eps_m^*$. 
The convexity of the function to the left of $\eps_m^*$ guarantees the 
monotonicity of the sequence and its boundedness\footnote{A much more accurate approximation 
is obtained by the modified iteration
$\tilde\eps_{k+1} = \eps_{k} - 2{f(\eps_k)}/{f'(\eps_k)}$,
which is such that $|\tilde\eps_{k+1} - \eps_m^*| \approx {\rm const} |\eps_{k} - \eps_m^*|^2$; 
see \cite{GL17}.}.
We refer the reader to \cite{GL17} for full details.

{
\begin{remark}
Assume that for $\eps < \eps_m^*$, $\eps \rightarrow \eps_m^*$, $F_\eps(E(\eps)) \rightarrow 0$,
 and exactly $m$ eigenvalues of $\sym\left(A+\eps B E(\eps)\right)$ vanish.
Let $E_* = \lim\limits_{\eps \rightarrow \eps_m^*} E(\eps)$.
Then, exploiting Theorem \ref{thm:stat} and the rank-properties of $E(\eps)$, and passing 
to the limit it follows that $E_*$ has the form
$$
E_* = Z D X^{T}
$$
with $D$ a diagonal matrix and the orthonormal columns of
$X$  span the invariant space of $\sym(A -BK_*)$  associated with the
$m$ rightmost (zero) eigenvalues, and $Z=B^T X$. Therefore, $E_*= B^T X D X^T$ and
$K_* = \eps_m^* E_*$ has rank-$m$.
%
 
If instead $m' > m$ eigenvalues effectively vanish, and $m_+ \le q$ then $E_*$ has rank $m_+$.
\end{remark}


With a Frobenius norm minimizing feedback matrix $K_*=\varepsilon_m^* E_*$ we thus have that
the matrix $A-\varepsilon_m^* BB^T (XDX^T)$ provides a {\it dissipative} closed-loop  system. 
This reminds us of a corresponding property of
the solution $X_*$ to the Riccati equation, and in particular, that
$A-BB^T X_*$ is associated with a {\it stable} closed-loop system.

\subsection{A variant of the gradient system approach: a modified functional}
\label{sec:plus}

The proposed functional \eqref{F-eps} is not the only possible one. Here we shortly describe a
variant that has been shown to be more effective in our experiments. Note that the associated
gradient system has a very similar structure, although the gradient in this case is only continuous.

We use the notation $a^+ = \max\left\{a,0\right\}$. For a fixed $\eps>0$ we consider the 
minimization of the following function
\begin{equation}
\label{F-epsp}
F^+_\eps(E) = \frac12 \sum_{i=1}^m \Bigl( \lambda_i^+\left( \sym(A -\eps B E) \right)  \Bigr)^2 ,
\qquad\quad \hbox{constrained by }\|E\|_F=1.
\end{equation}
The free gradient is continuous and has the form
\begin{equation} 
G^+(E)={}-\sum_{i=1}^{m^+(E)}  \lambda_i z_i x_i^{T}, \qquad z_i = B^{T} x_i
\label{eq:gradp}
\end{equation}
where $m^+(E) \le m$ is the number of positive eigenvalues among the $m$ rightmost ones.
This means that negative eigenvalues (among the $m$ largest) do not contribute to the gradient
which has rank equal to $m^+$. This modified strategy, which we shall call GL($m$)+ in our
numerical experiments, is able to account for more strongly varying eigenvalues, that 
possibly cross the origin while converging to zero as the iterations proceed.

\begin{remark}
An important advantage of \eqref{F-epsp} is that it no longer depends on $m$,
but only on $m^+(E)$. In particular, 
if $m$ is larger than the number of positive eigenvalues of $\sym(A -\eps B E)$ 
during the whole optimization process, the method is expected to converge. 
This also means that whenever using GL($m$)+, 
by taking a sufficiently large $m$ we expect to obtain the same results, independently 
of $m$  (see Example \ref{ex:smallA_B3}). Only if $m$ is chosen smaller than the final number of 
eigenvalues coalescing to zero we should expect an incorrect behavior. If non-convergence
is observed, then one can readily increase the value of $m$.
\end{remark}

\section{Numerical experiments}\label{sec:expes}
In this section we report on some of our computational experiments for determining
the minimum norm feedback matrix. In particular, we analyze the
behavior of the different methods we have discussed, with special
emphasis on the minimization property, using both the Frobenius and the Euclidean norms.
In all examples, we checked a-priori that the system can be made dissipative, that is
Theorem \ref{th:exist} holds.

The methods we are going to investigate are summarized as follows:

\vskip 0.1in
\begin{center}
\begin{tabular}{ll}
Method & description \\
\hline
GL$(m)$     & two-step method of section \ref{sec:GL} with $m$ rightmost eigenvalues \\
LMI    & Matlab basic function for the LMI problem (\ref{eqn:LMI_2norm}) ({\tt mincx}) \\
Yalmip1 & Matlab version of Yalmip with SeDuMi solver for problem (\ref{eqn:LMI_2norm})\\
Yalmip2 & Matlab version of Yalmip with SeDuMi solver for problem (\ref{eqn:LMI_Fnorm})\\
Pencil  & minimization problem with pencil in (\ref{Opt1}) \\
\hline
\end{tabular}
\end{center}
\vskip 0.1in

\begin{example}\label{ex:smallA_B2}
{\rm
We consider the following small data set
\begin{equation} \label{ex:ill1}
A = \begin{bmatrix}
   -0.2  &    1.6   &   0.2   &   2.6   &  -0.4 \\  
   -0.2  &   -0.8   &  -1.2   &  -0.7   &  -1.8 \\  
    1.4  &    0.7   &  -1.1   &   0.2   &   0.8 \\ 
    0.3  &    0.8   &   0.1   &  -0.1   &  -0.9 \\ 
    0.2  &   -0.2   &   0.7   &  -1.9   &   0.1  
\end{bmatrix}, \qquad
B = \begin{bmatrix}
    0.6  &    0.5 \\  
   -0.2  &    0.3 \\  
    0.5  &      0 \\
    0.2  &    0.6 \\ 
    0.6  &   -0.6  
\end{bmatrix} .
\end{equation}
The eigenvalues of the matrix $(A+A^T)/2$ 
are given by (with $4$ decimal digits) 
$\{-2.4752,\,  -1.8301,\,  -0.7238,\, 0.6506,\, 2.2785 \}$,
including two $2$ positive eigenvalues.
The performance of the considered methods is reported in Table \ref{tab:ex1_res}.

The GL method was used with $m=2$. The dissipating matrices for GL and Yalmip2
are, respectively
$$
K_{GL} =
 \begin{bmatrix}
   0.3690& -0.12149&  0.34503&  0.1119&  0.35065 \\
   1.0340&  0.66501& -0.01895&  1.3640& -1.2432 \\
 \end{bmatrix}
$$
and
$$
K_{Yalmip2} =
 \begin{bmatrix}
   0.3684 &-0.11954&  0.35079&  0.1097&  0.3467 \\
   1.0118 & 0.65736& -0.03002&  1.3995& -1.2240 \\
 \end{bmatrix}
$$
showing that the two matrices are not the same, even accounting for
numerical approximations. Similarly, for the eigenvalues of the
symmetric parts of the dissipative matrix we obtain
$$
\lambda_i({\sym}(A-B K_{GL})) \in \{
  -2.4765,\,  -1.8306,\,  -0.72468,\,  -2.4e-09,\,  -1.3e-08\} ,
$$
and
$$
\lambda_i({\sym}(A-B K_{Yalmip2})) \in \{
  -2.4743,\,  -1.8298,\,  -0.72353,\,  -2.4e-10,\,   5.0e-10\} .
$$
Notice that because of finite precision arithmetic - the quantities actually
minimized are the squares of the ones sought after -
neither method is able to force the two eigenvalues to zero to machine precision.
We also observe that
$$
\lambda_i({\sym}(A-B K_{Yalmip1})) \in \{
  -2.4742,\,  -1.8280,\,  -0.72428,\,  -0.69001,\,   2.9e-11\}
$$
that is, the minimization of the 2-norm correctly moves both
positive eigenvalues of $A+A^T$, but only one is moved to zero.
It is also interesting to notice that in all cases, the negative eigenvalues
of $(A+A^T)/2$ are barely moved.

\begin{table}
\centering
\begin{tabular}{|lcrr|}
Method & Minimization &  $\|K_*\|_2$ & $\|K_*\|_F$ \\
\hline
GL$(2)$     & F-norm       &   {\bf 2.2166} & {\bf 2.3063}  \\
LMI    &  2-norm      &   {\bf 2.2166} &   2.6714 \\
Yalmip1 &  2-norm     &   {\bf 2.2166} &   2.5765 \\
Yalmip2 &  F-norm     &   {\bf 2.2166} &   {\bf 2.3063} \\
Pencil  &  F-norm     &   2.2560  &  2.7585  \\
\hline
\end{tabular}
\caption{Example \ref{ex:smallA_B2}.\label{tab:ex1_res}}
\end{table}
 
Finally, {\color{black}the upper plot of} Figure \ref{fig:ex1} shows
the field of values $W(A-B K_{Yalmip2})$ or
$W(A-B K_{GL})$, as they are visibly indistinguishable.
{\color{black} The multiple zero eigenvalue of Sym($A-BK_{GL}$) causes
a flat portion of the right boundary of $W(A-B K_{GL})$. It can be shown that
the flat segment is given by $[-\imath \sigma, \imath \sigma]$, where
$\sigma$ is the spectral radius of $\skews(A-BK_{GL})$ restricted
to the kernel of $\sym(A-BK_{GL})$. We refer the
reader to, for instance, \cite{ERS12} and its
references for a more detailed account on flat portions on the
boundary of the field of values.
The lower plot of Figure \ref{fig:ex1} reports the field of value of
$A-B K_{Yalmip1}$: the simple zero eigenvalue of Sym($A-BK_{Yalmip1}$) determines
a more curved boundary on the right. }
$\square$
}
\end{example}

\begin{figure}[htb!]
\centering
\includegraphics[scale=0.66]{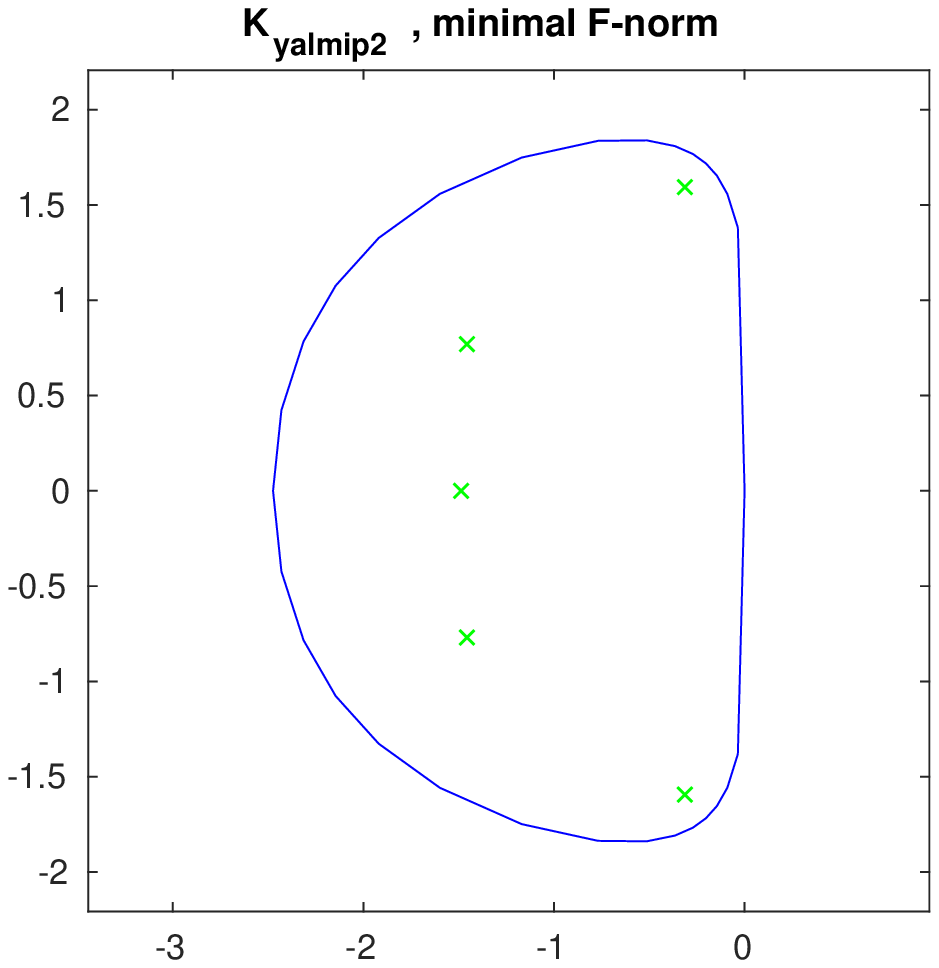} \vskip 5mm
\includegraphics[scale=0.66]{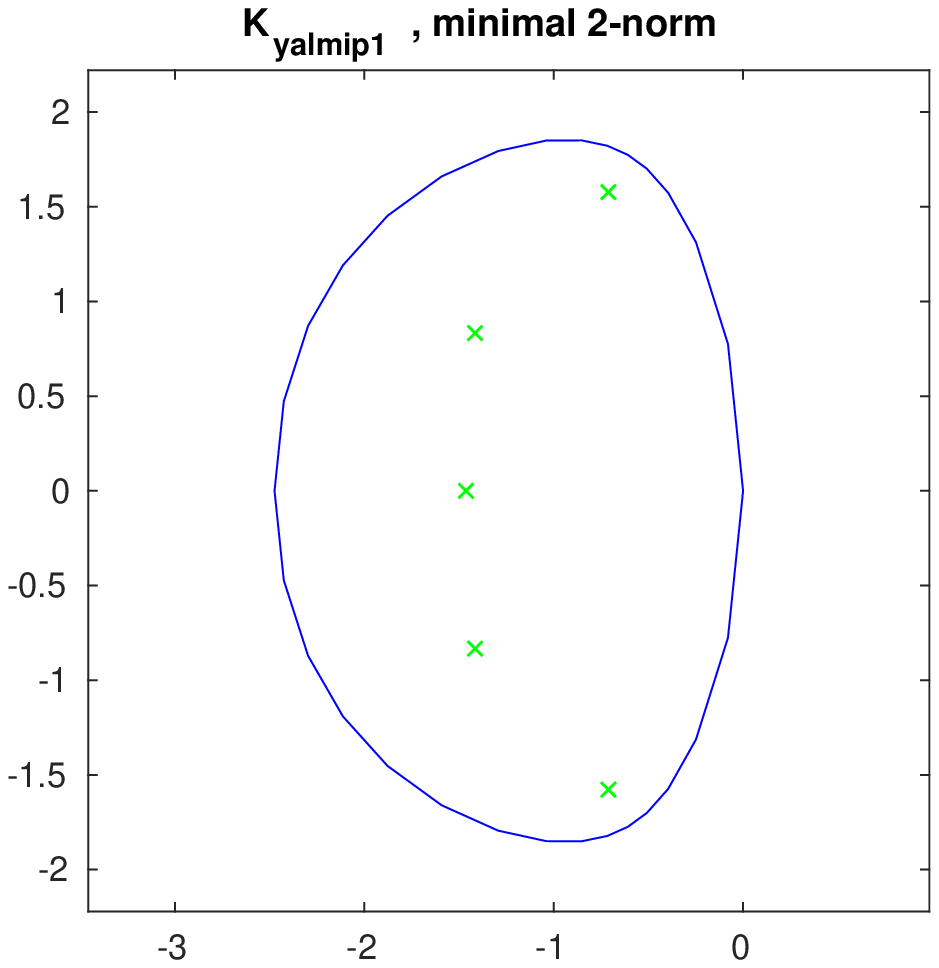}
\caption{Example \ref{ex:ill1}.  Upper: $W(A-B K_{GL})$.  Lower: $W(A-B K_{Yalmip1})$.  \label{fig:ex1}}
\end{figure}

{\color{black}
As a consequence of the discussion in the previous example, 
in the case one expects a weakly passivating feedback leading to
a multiple eigenvalue of $\sym(A-BK)$, it is useful 
{\color{black}to enforce passivity of the feedback system by (slightly) shifting to the left the right
boundary of $W(A-BK)$. {\color{black} This could be achieved, for instance, }
by solving the optimization problem P2
for the matrix $A+\delta \Id$ (instead of $A$) for a suitable small $\delta>0$.
This would provide a strictly dissipating feedback with the field of values $W(A-B K)$  
characterized by a flat right boundary along a line parallel to the imaginary axis, 
passing through the point $(-\delta,0)$.
}

For the sake of comparisons, in the following we shall focus only
on the two Frobenius norm minimizing methods.}

\begin{example}\label{ex:smallA_B3}
{\rm
Consider again the matrix $A$ of \eqref{ex:ill1} but consider now the augmented matrix $B$
\begin{equation}
B = \left( \begin{array}{rrr}
    0.6  &    0.5 & 1 \\  
   -0.2  &    0.3 & 0 \\  
    0.5  &      0 & 0 \\
    0.2  &    0.6 & 0 \\ 
    0.6  &   -0.6 & 0 
\end{array} \right).
\label{ex:ill1b}
\end{equation}

The results for this new $B$ are displayed in Table \ref{tab:ex2_res}, and
they are similar to those of the previous test, in spite of the larger $B$.
In this example, we also report on the behavior of GL for a different number
$m$ of eigenvalues to be moved to zero. For $m=2$ (the number of positive
eigenvalues of $(A+A^T)/2$) both norms are smaller than for $m=q=3$.
The results in Table \ref{tab:ex2_res} show that for GL is important to capture
the actual number of positive eigenvalues of $(A+A^T)/2$ to obtain a
close-to-optimal feedback matrix.
$\square$
}
\end{example}

\begin{table}
\centering
\begin{tabular}{|lcrr|}
Method & Minimization &  $\|K_*\|_2$ & $\|K_*\|_F$ \\
\hline
GL(2)     & F-norm    &    {\bf 2.0713}  &   {\bf 2.1476}  \\
GL(3)     & F-norm    &    2.3699        &   3.0638   \\
LMI    &  2-norm      &   {\bf 2.0705} &   2.5668 \\
Yalmip1 &  2-norm     &   {\bf 2.0705} &   2.3946 \\
Yalmip2 &  F-norm     &   {\bf 2.0713} &   {\bf 2.1476} \\
Pencil  &  F-norm     &   3.6459       &  3.9537  \\
\hline
\end{tabular}
\caption{Example \ref{ex:smallA_B3}. Here GL($m$) means that $m$ eigenvalues were
moved to zero in the minimization problem (\ref{eqn:minmin}). \label{tab:ex2_res}}
\end{table}

\begin{figure}[htb]
\centering
\includegraphics[scale=0.5]{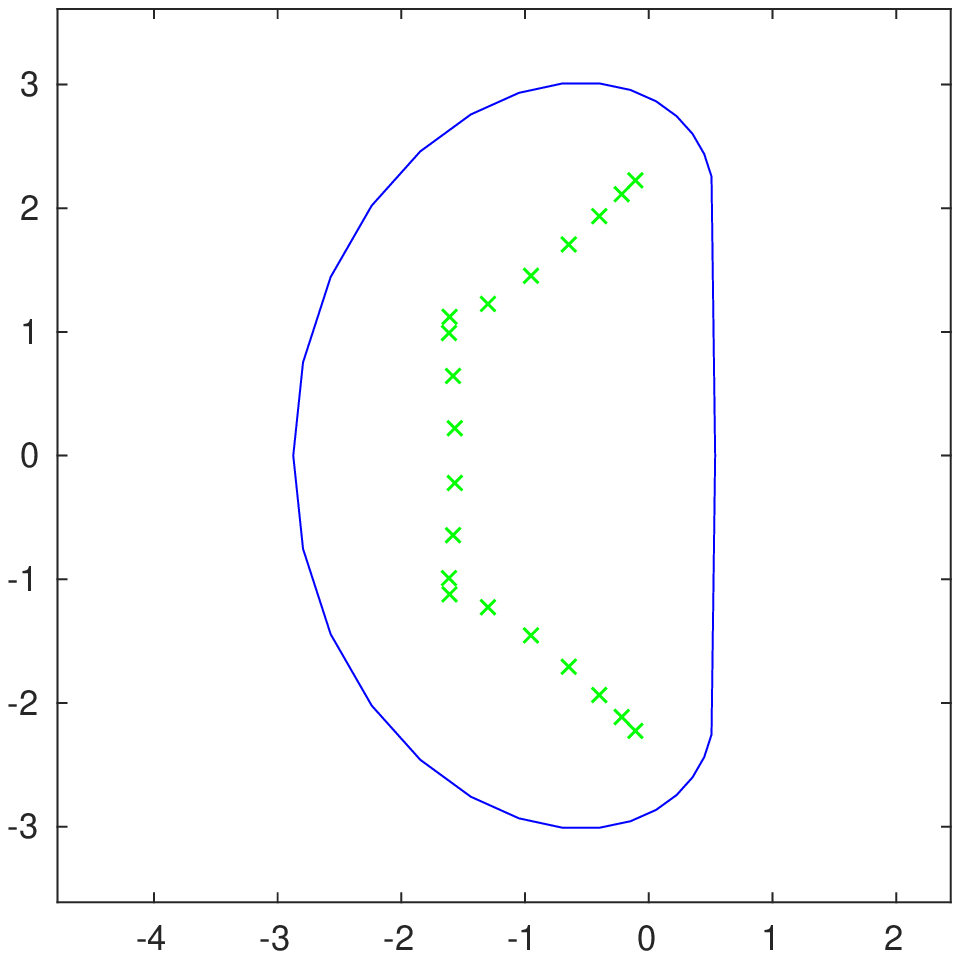}
\caption{Spectrum and field of values of the matrix Grcar matrix ($n=20$) in Example \ref{ex:grcar}. \label{fig:grcar}
}
\end{figure}

\begin{example}\label{ex:grcar}
{\rm
We consider the negative Grcar matrix of size $n$, defined as a Toeplitz banded matrix
with unit lower bandwidth of elements equal to minus one, and upper bandwidth three, given
by all ones. Its spectrum and field of values are
given in Figure \ref{fig:grcar} for $n=20$. The symmetric part of the
original matrix has a large number of positive eigenvalues, so that a shifting
procedure is adopted to have $m=O(1)$ positive eigenvalues.  To ensure that dissipation 
is feasible $B$ was selected as a linear combination of all eigenvectors corresponding
to positive eigenvalues of $(A+A^T)/2$, so that $q=m$. 

\bng
{\color{black}Although the considered matrix $B$ makes the problem strongly non-generic, it is illustrative}
of a situation where the GL($m$) method performs critically.
\eng 
The results of using GL($m$) and Yalmip2 are reported in Table \ref{tab:grcar},
as $n$ and the shift vary. The reported values show that the two methods approximately
return the same minimum, with Yalmip2 always being smaller. 
\bng Indeed the higher accuracy obtained by Yalmip is not unexpected since it makes use 
of a Newton method, {\color{black}whereas} GL($m$) is based on a gradient method.
\eng
It is interesting that
is some cases (incidentally corresponding to $m=4$) the discrepancy is slightly higher.
A closer look reveals that for these data the positive eigenvalues occur in pairs
of near eigenvalues. This seems to affect the performance of GL($m$). This anomalous,
though not fully unexpected behavior is explored in the next example. $\square$
}
\end{example}

\begin{table}[hbt]
\centering
\begin{tabular}{|lccll|} 
\hline
$A - 0.6I$ & $n$ & $m$  &    GL($m$) & Yalmip2 \\ 
\hline
& 50&   2    &       3.499028e-02   &        3.498990e-02 \\
&100&   4    &       7.339794e-01   &        7.291499e-01 \\
&150&   6    &       6.275579e-01   &        6.257247e-01 \\
&200&  10    &       2.448407e-01   &        2.448246e-01 \\
\hline
$A-0.62I$
&100&   2    &       2.181123e-02   &        2.181135e-02 \\
&150&   4    &       2.904286e-01   &        2.881408e-01 \\
\hline
$A-0.52I$
& 20&   2    &       1.627621e-02   &        1.627676e-02 \\
& 40&   3    &       3.019605e-01   &        3.019597e-01 \\
& 45&   4    &       1.931760       &        1.914460     \\
& 50&   4    &       2.275207       &        2.257378     \\
&100&   8    &       7.909541e-01   &        7.909255e-01 \\
&150&  13    &       6.278783e-01   &        6.278735e-01 \\
\hline
\end{tabular}
\caption{Example \ref{ex:grcar}. Minimum F-norm obtained by the considered methods,
as the Grcar matrix dimension varies, for different shifts. Here $q=m$. \label{tab:grcar}}
\end{table}

\begin{example}\label{ex:clusters}
{\rm
To deepen our understanding of the behavior of GL($m$) in case of positive clusters
we consider the following class of matrices
$$
{\cal A}=\frac 1 2 (A+A^T) = X \Lambda X^T, \quad \Lambda = 
{\rm diag}(\lambda_1, \ldots, \lambda_{n-q}, \eta_1, \ldots, \eta_q)
$$
where $\lambda_i$ are uniformly\footnote{Linear or logarithmic distributions yield similar results;
we used linearly distributed eigenvalues.} distributed eigenvalues in $[-10, -10^{-2}]$ while
$\eta_j \in \{1, 1+\delta, 2, 2+\delta, 3, 3+\delta\}$, taken in this order as $q$ varies,
so that positive clusters arise. $X$ is taken as a fixed orthonormal matrix, while
$\delta\in (0,1)$ varies, so as to increase the eigenvalue clustering.
The matrix $A$ is then obtained as the lower triangular part of ${\cal A}$, so that ${\cal A}=(A+A^T)/2$ holds.
The matrix size is $n=20$ throughout. The matrix $B$ was taken as in the previous
examples, so that $m=q$.

\begin{table}[hbt]
\centering
\begin{tabular}{|llcrr|}
\hline
 $m(=q)$  &   $\delta$ & GL($m$) &   GL($m$)+  & Yalmip2 \\
\hline
 6  &   0.00001  & - &  5.581468e+01    &    5.581342e+01  \\
 6  &   0.001    & - &  5.582551e+01    &    5.582426e+01  \\
 6  &   0.01     & - &  5.592403e+01    &    5.592278e+01  \\
 6  &   0.1      & - &  5.690962e+01    &    5.690837e+01  \\
 6  &   0.5      & 6.131648e+01 &  6.129619e+01    &    6.129493e+01  \\
 2  &   0.001    & 2.429389e+00 &  2.429389e+00    &    2.429388e+00  \\
 4  &   0.001    & - &  5.152558e+01    &    5.152495e+01  \\
 4  &   0.01     & - &  5.157901e+01    &    5.157837e+01  \\
 4  &   0.1      & - &  5.211364e+01    &    5.211302e+01  \\
 4  &   0.5      & - &  5.449942e+01    &    5.449883e+01  \\
\hline
\end{tabular}
\caption{Example \ref{ex:clusters}. Minimum F-norm obtained by the considered methods,
as the closeness and number of positive eigenvalues vary.
Here $m=q$. 
\label{tab:clusters}}
\end{table}
Table \ref{tab:clusters} shows the results of the considered methods, minimizing the Frobenius
norm. We vary both the number of positive eigenvalues of ${\cal A}$ and their closeness,
by tuning $\delta$.
We readily see that the LMI method Yalmip2 succeeds in determining the minimum, whereas
GL($m$) fails to converge in all but two cases, illustrating that the
method is indeed affected by this data setting.
The reason of this failure is that when (in the gradient dynamics)  the $m$-th largest
eigenvalue moves to the left of the uncontrollable eigenvalue $\lambda=-10^{-2}$ of $\sym(A)$
(the associated eigenvector $x$ is in fact such that $B^{\rm T} x = 0$), 
we have that $\lambda=-10^{-2}$ 
replaces such an eigenvalue in the functional and cannot be moved to $0$.  
\bng
Although this is a strongly non-generic case we can {\color{black} expect} 
that almost uncontrollable eigenvalues may slow down the speed of GL($m$). 
\eng




This problem can be effectively 
solved by the variant GL($m$)+ introduced in section~\ref{sec:plus};
Experiments with 
GL($m$)+ were thus included in Table~\ref{tab:clusters}.
We observe that this modification provided a dramatic improvement to the method, which converged to
practically the same value obtained with Yalmip2 in all cases. As this variant appears
to be new, its theoretical properties still need to be analyzed; we postpone this interesting
study to future research.  $\square$
}
\end{example}

Our experience on larger data showed that GL($m$) is 
faster than all LMI-based methods for medium to large values of $n$.
This is not unexpected, since the extremely high
computational cost is one of the known drawbacks of LMI-based algorithms. 
Although a CPU time comparison is not the focus of this paper, which would possibly require 
moving to compiled languages, we believe that there is enough numerical evidence to
encourage further exploration of GL($m$) and its variants towards an efficient
treatment of large scale problems.

%
\section{Conclusions}
Passivating matrices are of interest for open-loop dynamical systems and
have thus been analyzed in the Control literature. We have shown that
their classical parametrization may not include all possible such matrices,
and we have provided richer parametrization sets.

The problem of determining the norm minimizing dissipating feedback matrix can
be formulated as a linear matrix inequality problem, and thus solved
with well established software in the small size case.
We also explored a variant of a recently developed functional minimization
method, GL($m$), that appears to be able to determine the solution at a comparable
accuracy, with possibly lower computational efforts on
medium and large size problems.
In spite of these encouraging results, 
our numerical experiments also show that this new strategy requires further
theoretical and experimental investigations to be considered as an
effective viable alternative to LMI methods, and this will be the
topic of our future research.

\section*{Acknowledgments}
We thank Christopher Beattie (Virginia Tech) and Volker Mehrmann (TU Berlin) for
insightful conversations on topics related to this work.
We also thank the Italian INdAM GNCS for support. 

\section*{Appendix A}
In this Appendix we include the proof of Proposition \ref{prop:nonunique}.

\begin{proof}
Let us partition $Q=[Q_1, Q_2]$ conforming to the partitioning of $\Lambda_{\pm}$,
and define $Z = Q_1 H_1 + Q_2 H_2 \in\RR^{(n+q)\times n}$  for some 
$H_1\in\RR^{n\times n}$ nonsingular and $H_2\in\RR^{m\times n}$.
Then
\begin{eqnarray*}
Z^T {\cal M} Z &=& Z^TQ_1 \Lambda_+ Q_1^T Z - Z^TQ_2 |\Lambda_-| Q_2^T Z \\
&=& 
H_1^T \Lambda_+ H_1 - H_2^T |\Lambda_-| H_2 
= H_1^T ( \Lambda_+ - (H_2 H_1^{-1})^T |\Lambda_-| H_2 H_1^{-1}) H_1.
\end{eqnarray*}
Let $\hat H =  H_2 H_1^{-1}$, and denote with
$\lambda_{\min}^+$ the smallest eigenvalue of $\Lambda_+$, and with
$|\lambda_{\max}^-|$ the largest eigenvalue of $|\Lambda_-|$. 
Then, for any $0\ne x\in\RR^{n}$ and $y=H_1 x$ (note that
$y\ne 0$ due to the nonsingularity of $H_1$), we can write
$$
x^T Z^T {\cal M} Z x = y^T (\Lambda_+ - \hat H^T |\Lambda_-|  \hat H) y  \ge
(\lambda_{\min}^+ - \|\hat H\|^2 |\lambda_{\max}^-|) \|y\|^2.
$$
If $H_1, H_2$ are chosen so that $\alpha = \|\hat H\|^2$ satisfies
$\lambda_{\min}^+ - \|\hat H\|^2 |\lambda_{\max}^-| > 0$, then
$Z^TMZ$ is positive definite.

We next show that $H_1, H_2$ can be chosen so that $Z$ has the
form $Z=[I;K]$ for some $K$. Let us further partition $Q$ as 
$$
Q = [Q_1, Q_2] = \begin{bmatrix} Q_{11} & Q_{12}\\ Q_{21} & Q_{22} \end{bmatrix},
$$
so that $Z = [Q_{11};Q_{21}]H_1 + [Q_{12};Q_{22}] H_2$. Note that $Q_{11}$
is nonsingular, for the proof of the previous theorem.
 We then impose the structure of $Z$, that is
$$
\begin{bmatrix}Q_{11}\\Q_{21}\end{bmatrix} H_1 + 
\begin{bmatrix}Q_{12}\\Q_{22}\end{bmatrix} H_2 = 
\begin{bmatrix}I\\K\end{bmatrix}, \quad
Q_{11} H_1 + Q_{12} H_2 = I, \, K = Q_{21} H_1 + Q_{22} H_2.
$$
It follows that $H_1 = Q_{11}^{-1}(I - Q_{12}H_2)$. Therefore, for any $H_2$ such
that $I-Q_{12}H_2$ is nonsingular, the matrix $H_1$ is nonsingular, and
$K$ is well defined. The statement is proved by choosing 
$H_2$ so that $\alpha=\|H_2H_1^{-1} \|^2 = \|H_2(I-Q_{12}H_2)^{-1}Q_{11}\|^2$, with $\alpha$ satisfying 
$\lambda_{\min}^+ - \alpha |\lambda_{\max}^-| > 0$.

Finally, substituting $H_1$ is the relation $K = Q_{21} H_1 + Q_{22} H_2$ and collecting
terms we obtain
$K = Q_{21} Q_{11}^{-1} + (Q_{22}-Q_{21}Q_{11}^{-1} Q_{12}) H_2$.
\end{proof}

\section*{Appendix B}
In this Appendix we prove Proposition \ref{prop:rankt}.

\begin{proof}
The hypotheses ensure that $-(A +A ^T) + BK + K^TB^T > 0$. 
Let us introduce the following eigenvalue decomposition 
$$
{\cal J}:=\begin{bmatrix}0 & I \\ I & 0\end{bmatrix} =
\frac 1 2 \begin{bmatrix}I & I \\ I & -I\end{bmatrix} 
\begin{bmatrix} I & 0 \\ 0 & -I\end{bmatrix} 
\begin{bmatrix} I & I \\ I & -I\end{bmatrix}  =: W {\cal I} W^T.
$$
 Moreover, letting $[U_+, U_-] := [B, K^T] W =\sqrt{\frac 1 2} [B+K^T, B-K^T]$
we have
\begin{eqnarray*}
0 < -(A +A ^T) + BK + K^TB^T 
&=& -(A +A ^T) + [B, K^T] {\cal J} \begin{bmatrix}B^T\\K \end{bmatrix}\\
&=& -(A +A ^T) + [B, K^T] W {\cal I} W^T  \begin{bmatrix}B^T\\K \end{bmatrix} \\
&=& -(A +A ^T) + [U_+,U_-] {\cal I} [U_+,U_-]^T .
\end{eqnarray*}
Therefore, letting $\Lambda_-\in\RR^{t\times t}$ denote the negative eigenvalue matrix
of $-\sym(A )$ and
multiplying from both sides by $Q_-$,
\begin{eqnarray*}
0 < Q_-^T(-(A +A ^T) + BK + K^TB^T )Q_- 
&=& \Lambda_- + Q_-^T [U_+, U_-] {\cal I} [U_+, U_-]^T Q_- \\
&=& \Lambda_- - Q_-^T U_- U_-^T Q_- + Q_-^T U_+ U_+^TQ_- .
\end{eqnarray*}
Here the term $Q_-^T U_+ =\sqrt{\frac 1 2} Q_-^T(B+K^T)$ has dimensions $t\times q$.
Finally, we notice that the first two terms in the last
expression are negative definite, so
that, 
to satisfy
the positivity constraint the matrix $Q_-^T U_+ U_+^T Q_-$ must
move all $t$ eigenvalues of $\Lambda_- - Q_-^T U_- U_-^T Q_-$ to
the non-negative half real axis. In particular, its rank must be at least $t$.
\end{proof}

\end{document}